\documentclass[12pt,a4paper]{amsart}

\usepackage{amssymb}
\usepackage{amscd}

\usepackage{hyperref}

\usepackage[numeric, abbrev, nobysame]{amsrefs}

\usepackage[all]{xypic}

\def\frak{\mathfrak}
\def\Bbb{\mathbb}
\def\Cal{\mathcal}

\let\phi\varphi

\newcommand{\x}{\times}

\newcommand{\al}{\alpha}
\newcommand{\be}{\beta}
\newcommand{\ga}{\gamma}

\newcommand{\la}{\lambda}
\newcommand{\om}{\omega}
\newcommand{\ph}{\phi}
\newcommand{\ps}{\psi}
\renewcommand{\th}{\theta}

\newcommand{\Ga}{\Gamma}
\newcommand{\La}{\Lambda}
\newcommand{\Ph}{\Phi}

\newcommand{\Om}{\Omega}

\newcommand{\im}{\operatorname{im}}
\renewcommand{\Im}{\operatorname{im}}
\newcommand{\id}{\operatorname{id}}
\newcommand{\ad}{\operatorname{ad}}

\newcommand{\tr}{\operatorname{tr}}

\newcounter{theorem}
\numberwithin{theorem}{section}
\newtheorem{thm}[theorem]{Theorem}
\newtheorem*{thm*}{Theorem \thesubsection}
\newtheorem{lemma}[theorem]{Lemma}
\newtheorem{prop}[theorem]{Proposition}
\newtheorem{cor}[theorem]{Corollary}
\newtheorem*{lemma*}{Lemma \thesubsection}
\newtheorem*{prop*}{Proposition \thesubsection}
\newtheorem*{cor*}{Corollary \thesubsection}

\theoremstyle{definition}
\newtheorem{definition}[theorem]{Definition}
\newtheorem*{definition*}{Definition \thesubsection}
\newtheorem{example}[theorem]{Example}
\newtheorem*{example*}{Example \thesubsection}
\theoremstyle{remark}
\newtheorem{remark}[theorem]{Remark}
\newtheorem*{remark*}{Remark \thesubsection}

\def\sideremark#1{\ifvmode\leavevmode\fi\vadjust{\vbox to0pt{\vss
 \hbox to 0pt{\hskip\hsize\hskip1em
 \vbox{\hsize3cm\tiny\raggedright\pretolerance10000
  \noindent #1\hfill}\hss}\vbox to8pt{\vfil}\vss}}}%
                        
\begin{document}

\title{Parabolic conformally symplectic structures I;\\ 
definition and distinguished connections} 
\date{September 27, 2017}
\author{Andreas \v Cap and Tom\'a\v s Sala\v c}
\thanks{Support by projects P23244--N13 (both authors) and P27072--N25
  (first author) of the Austrian Science fund (FWF) is gratefully
  acknowledged.} 

\address{A.\v C.: Faculty of Mathematics\\
University of Vienna\\
Oskar--Morgenstern--Platz 1\\
1090 Wien\\
Austria}
\address{T.S.: Mathematical Institute\\ Charles University\\ Sokolovsk\'a
  83\\Praha\\Czech Republic}
\email{Andreas.Cap@univie.ac.at}
\email{salac@karlin.mff.cuni.cz}

\begin{abstract}
  We introduce a class of first order G--structures, each of which has
  an underlying almost conformally symplectic structure. There is one
  such structure for each real simple Lie algebra which is not of type
  $C_n$ and admits a contact grading. We show that a structure of each
  of these types on a smooth manifold $M$ determines a canonical
  compatible linear connection on the tangent bundle $TM$. This
  connection is characterized by a normalization condition on its
  torsion. The algebraic background for this result is proved using
  Kostant's theorem on Lie algebra cohomology.

  For each type, we give an explicit description of both the geometric
  structure and the normalization condition. In
  particular, the torsion of the canonical connection naturally splits
  into two components, one of which is exactly the obstruction to the
  underlying structure being conformally symplectic. This article is
  the first in a series aiming at a construction of differential
  complexes naturally associated to these geometric structures.
\end{abstract}

\subjclass[2010]{Primary: 53D15, 53B15, 53C10, 53C15; 
Secondary: 53C29, 53C55}

\keywords{almost conformally symplectic structure, special symplectic
  connection, exceptional symplectic holonomy, finite order geometric
  structure, first prolongation, canonical connection}

\maketitle

\pagestyle{myheadings} \markboth{\v Cap and Sala\v c}{PCS--structures
  I}

\section{Introduction}\label{1}
This article is the first part in a series of three. The main
motivation for this series originally came from the work \cite{E-G} of
M.G.~Eastwood and H.~Goldschmidt on integral geometry. The main part
of that article is devoted to the construction of a family of
differential complexes on complex projective space, and to proving
some results on their cohomology, which then imply results on integral
geometry. The form of these complexes is rather unusual and the
construction in \cite{E-G} does not explain whether these complexes
are associated to a geometric structure on $\Bbb CP^n$ and, if yes,
what this structure actually is.

An attempt to sort out these questions was made in the first version
of the preprint \cite{E-S} by M.G.~Eastwood and J.~Slov\'ak. There the
authors define so--called conformally Fedosov structures and associate
a tractor bundle to such a structure. In the second version of the
preprint, which has appeared very recently, this tractor bundle was
used to construct differential complexes of the kind used in
\cite{E-G}. This construction starts from twisted de--Rham complexes
associated to tractor bundles and is similar to the machinery of
BGG--sequences as introduced in \cite{CSS-BGG} and \cite{CD}.

The tractor bundle associated to a conformally Fedosov structure in
\cite{E-S} actually looks similar to the standard tractor bundle of a
contact projective structure in one more dimension. This leads to the
idea of constructing sequences like the ones from \cite{E-G} via
descending usual BGG sequences from a contact projective structure in
one higher dimension (which also is in accordance with the length of
the complexes that can be traced from the construction in
\cite{E-G}). Now contact projective structures fit into the class of
so--called \textit{parabolic contact structures}. These are finite
order geometric structures (indeed, parabolic geometries), which carry
an underlying contact structure. The available types of parabolic
contact structures correspond to simple Lie algebras which admit a
so--called contact grading (which is the case for almost all
non--compact real simple Lie algebras). Contact projective structures
in this picture correspond to Lie algebras of type $C_n$ and are
slightly exceptional, see Section 4.2 of \cite{book}.

In the series of articles starting with this one, we carry out the
idea of descending BGG complexes to appropriate quotients. This is not
only done for contact projective structures but for all parabolic
contact structures. It turns out that the quotients in question can be
characterized by the fact that they carry certain geometric
structures. The current article is devoted to the study of the basic
properties of these geometric structures, independently of any
realization as a quotient. The exceptional behavior of the $C_n$--type
Lie algebras mentioned above also shows up in this setting. Hence we
exclude them from the discussion this part of the series. We will
discuss conformally Fedosov structures in the framework of
contactification in the second part \cite{PCS2} of the series.

To any contact grading on a real simple Lie algebra, which is not
of type $C_n$, we associate a first order $G$--structure, which has an
underlying almost conformally symplectic structure. We call the
resulting structures \textit{parabolic almost conformally symplectic
  structures} or \textit{PACS--structures}. If the underlying
structure is conformally symplectic, then the structure is called a
\textit{PCS--structure}. It should be mentioned here that we do not
use the classical definition of an (almost) conformally symplectic
structure via a representative two--form. Rather than that, we proceed
similarly to conformal geometry and view the structures as line
subbundles in the bundle of two forms, which leads to several
simplifications.

The main result of the article is Corollary \ref{cor4.4}, which states
that any PACS--structure on a smooth manifold $M$ determines a unique
linear connection on $TM$ whose torsion satisfies a suitable
normalization condition. Moreover, the torsion of this canonical
connection naturally splits into two components. One of these is a
complete obstruction against the underlying almost conformally
symplectic structure being conformally symplectic, while the other is
an obstruction against integrability of the additional structure.

This main result is proved by showing that the Lie algebra of the
structure group of any PACS--structure has vanishing first
prolongation (in the standard sense of Sternberg, see
\cite{Sternberg}). This result is proved via Kostant's theorem (see
\cite{Kostant}) on Lie algebra cohomology, which also leads to an
appropriate normalization condition on the torsion. We also use
Kostant's theorem to show that, except in the $A_n$--case, the Lie
algebra of the structure group of a PACS--structure is a maximal
subalgebra of the corresponding conformally symplectic Lie
algebra. The computations of Lie algebra cohomologies needed for our
applications of Kostant's theorem are available in the literature on
parabolic contact structures, see Section 4.2 of \cite{book}, which
also provides explicit descriptions of the normalization conditions on
torsions. In that way, we obtain both an explicit description of all
PACS--structures (see Sections \ref{3.2} to \ref{3.5}) and explicit
descriptions and interpretations of the components of the torsion of
the canonical connections (see Sections \ref{4.5} and \ref{4.6}).

There is a second important motivation for this article. The
PCS--structures we introduce can be viewed as forming the geometric
background for the class of \textit{special symplectic connections} as
introduced by M.~Cahen and L.~Schwachh\"ofer in
\cite{Cahen-Schwachhoefer}. The latter form a class of torsion free
linear connections which preserve a symplectic form and satisfy a
certain condition on their curvature (which is related to contact
gradings of simple Lie algebras), see Section \ref{4.7} for
details. This class on the one hand contains all affine connections of
exceptional holonomy, which preserve a symplectic form. According to
the classification of affine holonomies, see \cite{M-S}, these cover a
substantial part of all exceptional holonomies. On the other hand, the
Levi--Civita connections of Bochner--K\"ahler metrics, see
\cite{Bryant}, and of Bochner--bi--Lagrangean metrics are special
symplectic connections. We prove in Theorem \ref{thm4.7} that, except for
connections of Ricci type (which correspond to type $C_n$), any
special symplectic connection is the canonical connection of a torsion
free PCS--structure. Moreover, if the type is also different from
$A_n$, the converse holds, i.e.~the canonical connection of a
torsion free PCS--structure is automatically a special symplectic
connection. In particular, all affine connections with exceptional
symplectic holonomy fall into this class.

For completeness, let us briefly describe the contents of the other
articles in the series. In the second part \cite{PCS2}, we describe
the relation between PCS--structures and parabolic contact
structures. This builds on the corresponding relation between
conformally symplectic structures and contact structures via
contactification as discussed in \cite{Cap-Salac}. In this context, we
also consider conformally Fedosov structures as introduced in
\cite{E-S} (adapted to the version of conformally symplectic
structures that we use). These can be considered as the analogs of
PCS--structures corresponding to the contact gradings of the simple
Lie algebras of type $C_n$. We show that a quotient of a parabolic
contact structure by a transversal infinitesimal automorphism inherits
a PCS--structure of the corresponding type. Moreover, locally any
PCS--structure arises in this way and the inducing parabolic contact
structure is locally unique up to isomorphism.

We also clarify the relation between the canonical connection
associated to the PCS--structure on the quotient and distinguished
connections of the original parabolic contact structure. Finally,
using contactifications, we complete the characterization of special
symplectic connections in terms of PCS--structures. This provides new
proofs and generalizations to cases of non--trivial torsion for several
results from \cite{Cahen-Schwachhoefer}.

In the last article \cite{PCS3} of the series, contactifications are
used to descend BGG sequences and relative BGG sequences associate to
parabolic contact structures to natural sequences of differential
operators on manifolds endowed with PCS--structures and study their
properties. In many situations this can be used to construct
differential complexes intrinsically associated to special symplectic
connections and more general PCS--structures. 

\section{Conformally symplectic structures}\label{2}

We start by looking at (almost) conformally symplectic structures from the
point of view of first order $G$--structures.

\subsection{Almost conformally symplectic structures}\label{2.1}
Traditionally, (locally) conformally symplectic structures are defined
via representative two forms. For our purposes, it will be more
natural to use notions suggested by the theory of conformal
structures. 

\begin{definition}\label{def2.1}
Let $M$ be a smooth manifold of even dimension $n=2m\geq 4$. 

(1) An \textit{almost conformally symplectic structure} on $M$ is a
smooth line subbundle $\ell\subset\La^2T^*M$ in the bundle of
two--forms on $M$ such that for each $x\in M$, each non--zero element
of $\ell_x$ is non--degenerate as a bilinear form on $T_xM$. 

(2) The structure is called \textit{conformally symplectic} if and
only if locally around each $x\in M$ there is a smooth section $\tau$
of $\ell$ which is closed as a two--form and satisfies $\tau(x)\neq
0$.
\end{definition}

Observe that in part (2), one may equivalently replace ``closed'' by
``exact''.

\begin{prop}\label{prop2.1}
Let $M$ be a smooth manifold of even dimension $n=2m\geq 4$ and let
$\ell\subset \La^2T^*M$ be an almost conformally symplectic
structure. Then we have

(i) If $\ell$ is conformally symplectic, then for any local
non--vanishing section $\tau$ of $\ell$, we have $d\tau=\ph\wedge\tau$
for some $\ph\in\Om^1(M)$.

(ii) Conversely, if $n>4$ and for each point $x\in M$, there is a local
section $\tau$ of $\ell$ which satisfies the condition in (i) and
$\tau(x)\neq 0$, then $\ell$ is conformally symplectic.

(iii) If $\ell$ is conformally symplectic, then local closed sections
of $\ell$ are uniquely determined up to a constant factor.
\end{prop}
\begin{proof}
(i) It suffices to show this locally, so we can assume that
  $\tau=f\tilde\tau$ for a nowhere vanishing closed section
  $\tilde\tau$ of $\ell$ and a non--zero function $f$. But then
  $d\tau=df\wedge\tilde\tau=(df/f)\wedge\tau$.

  (ii) If $\tau$ is a nowhere vanishing local section of $\ell$ such
  that $d\tau=\ph\wedge\tau$, then we conclude that
  $0=d\ph\wedge\tau$. For $n>4$, non--degeneracy of $\tau$ implies
  $d\ph=0$. Hence restricting to some smaller subset, we can find a
  smooth non--zero function $f$ such that $-\ph=d\log (f)=df/f$ and
  then $\tilde\tau=f\tau$ is closed.

(iii) As in (i), we write $\tau=f\tilde\tau$ with $\tilde\tau$ closed
  and non--vanishing. Then $0=d\tau$ implies $0=df\wedge\tilde\tau$, so
  non--degeneracy of $\tilde\tau$ implies $df=0$, even if $n=4$.
\end{proof}

Observe that for $n=4$, non--degeneracy of a two--form $\tau$ implies
that wedging with $\tau$ is an isomorphism from one--forms to
three--forms. Hence in this case, the condition from part (i) of the
proposition is always satisfied.

\subsection{First order G--structures}\label{2.2} 

Let us briefly review some basics from the theory of first order
G--structures. Consider a finite dimensional real vector space $V$ and
a Lie group $G$ endowed with an infinitesimally injective
representation on $V$. This means that one has given a homomorphism
$G\to GL(V)$ whose derivative is injective, thus identifying the Lie
algebra $\frak g$ of $G$ with a Lie subalgebra of $L(V,V)$. Then a
first order structure with structure group $G$ on a smooth manifold
$M$ with $\dim(M)=\dim(V)$ is defined as a smooth principal fiber
bundle $\Cal P\to M$ with structure group $G$, that is endowed with a
one--form $\th\in\Om^1(\Cal P,V)$. The form $\th$ has to be
$G$--equivariant and \textit{strictly horizontal} in the sense that
its kernel in each point is the vertical subbundle of $\Cal P\to
M$. This means that for each point $u\in\Cal P$ lying over $x\in M$,
the value $\th(u)$ descends to a linear isomorphism $T_xM\to V$, so
one obtains a map to the linear frame bundle of $M$. In particular,
$\th$ gives rise to an identification of the associated bundle $\Cal
P\x_GV$ with the tangent bundle $TM$.

The fundamental invariants of such structures are obtained via
connections. Recall that on any principal bundle there are principal
connections and that any principal connection on $\Cal P$ induces a
linear connection on the associated vector bundle $TM$. This induced
linear connection has a torsion which is a section of the bundle
$\La^2T^*M\otimes TM\cong\Cal P\x_G(\La^2V^*\otimes V)$. Now the
dependence of the torsion on the choice of connection can be described
using linear algebra, via the process of prolongation, see
\cite{Sternberg}.

As noted above, the Lie algebra $\frak g$ of $G$ is a subalgebra of
$L(V,V)=V^*\otimes V$. Then it is well known that the space of
principal connections on $\Cal P$ is an affine space modeled on the
space of smooth sections of the associated bundle $\Cal
P\x_G(V^*\otimes\frak g)$. The change of torsion induced by a
change of connection is described by a $G$--equivariant linear map
$$
\partial:V^*\otimes \frak g\to \La^2V^*\otimes V
$$ called the Spencer differential. Explicitly, this is given by first
including $V^*\otimes\frak g$ into $V^*\otimes V^*\otimes V$ and then
alternating in the first two arguments. Alternatively, viewing the
domain and target as linear maps and skew symmetric bilinear maps,
respectively, one has $\partial\Ph(v,w):=\Ph(v)(w)-\Ph(w)(v)$.

The subspace $\im(\partial)\subset L(\La^2V,V)$ gives rise to a smooth
subbundle in $\La^2T^*M\otimes TM$, and we denote by $\Cal J$ the
quotient of $\La^2T^*M\otimes TM$ by this subbundle. Then it is clear
from the above description that the projection of the torsion of the
induced connection on $TM$ to this quotient bundle is the same for all
principal connections on $\Cal P$. Hence one obtains a section of
$\Cal J$ which is an invariant of first order structures with
structure group $G$. This is called the \textit{intrinsic torsion} of
the structure.

On the other hand, $\frak g^{(1)}:=\ker(\partial)\subset
V^*\otimes\frak g$ is called the \textit{first prolongation} of $\frak
g$. The above discussion shows that for a fixed principal connection
$\ga$ on $\Cal P$, the space of all principal connections on $\Cal P$
which have the same torsion as $\ga$ is an affine space modeled on
sections of the associated bundle $\Cal P\x_G\frak g^{(1)}$. In
particular, if $\frak g^{(1)}=\{0\}$, then any principal connection on
$\Cal P$ is uniquely determined by its torsion.

The standard way to proceed further is to choose a $G$--invariant
linear subspace $N\subset \La^2V^*\otimes V$, which is complementary
to $\im(\partial)$. Usually, one refers to $N$ as a
\textit{normalization condition}. Via associated bundles, $N$
determines a smooth subbundle $\Cal N\subset\La^2T^*M\otimes TM$ and
from above we see that there always exist \textit{normal} principal
connections, i.e.~ones for which the torsion is a section of $\Cal
N$. The space of all normal connections is an affine space modeled on
$\Cal P\x_G\frak g^{(1)}$.

A classical simple example of this situation is the case $G=O(V)$, the
orthogonal group of a non--degenerate bilinear form on $V$. In this
case, first order structures with structure group $G$ are equivalent
to pseudo--Riemannian metrics on manifolds of dimension $n=\dim(V)$ of
the signature of the given bilinear form. One easily verifies that
$\partial$ is a linear isomorphism in this case, which shows that any
such geometry admits a unique torsion--free connection.

\subsection{}\label{2.3}
To apply this in the case of almost conformally symplectic structures,
consider a symplectic vector space $(V,b)$ of dimension $n=2m$ and
define
\begin{gather*}
Sp(V):=\{A\in GL(V):b(Av,Aw)=b(v,w) \quad \forall v,w,\in V\}\\
CSp(V):=\{A\in GL(V):\exists\la\in\Bbb R:b(Av,Aw)=\la b(v,w)  \quad
\forall v,w,\in V\}, 
\end{gather*}
the \textit{symplectic group} and the \textit{conformally symplectic
  group} of $V$. These are closed subgroups of $GL(V)$ and it is well
known that the Lie algebra $\frak{sp}(V)$ is simple, while
$\mathfrak{csp}(V)=\Bbb R\oplus\frak{sp}(V)$ is reductive with
one--dimensional center. Moreover, as a representation of $Sp(V)$, we
have $\mathfrak{sp}(V)\cong S^2V$, the symmetric square of the
standard representation $V$. 

The symplectic inner product $b$ determines a non--degenerate element
in $\La^2V^*$, whose inverse is a non--degenerate element
$b^{-1}\in\La^2V$. This gives rise to a $Sp(V)$--equivariant map
$\La^kV^*\to\La^{k-2}V^*$, which is surjective for $k\leq m+1$. Its
kernel is called the \textit{tracefree part}
$\La^k_0V^*\subset\La^kV^*$.

\begin{prop}\label{prop2.3}
Let $V$ be a real symplectic vector space of (even) dimension
$n$. Then we have:

(1) A first order structure with structure group $G:=CSp(V)$ is
  equivalent to an almost conformally symplectic structure.

  (2) For $\frak g:=\frak{csp}(V)$, we obtain $\frak g^{(1)}\cong
  S^3V$ and $(\La^2V^*\otimes V)/\im(\partial)\cong\La^3_0V^*$ as
  representations of $Sp(V)$. In particular $\partial$ is surjective
  for $n=4$.

(3) If $n>4$, then an almost conformally symplectic structure has
  vanishing intrinsic torsion if and only if it is conformally
  symplectic.

(4) For any torsion free connection compatible with a conformally
  symplectic structure, the induced connection on
  $\ell\subset\La^2T^*M$ is flat and its local parallel sections are
  exactly those sections of $\ell$ which are closed as two--forms on
  $M$. 
\end{prop}
\begin{proof}
  (1) The line in $\La^2V^*$ spanned by $b$ is by definition invariant
  under $G$, so it gives rise to a smooth line subbundle
  $\ell\subset\La^2T^*M$. Non--degeneracy of $b$ implies that any
  non--zero element in $\ell$ is non--degenerate, so it defines an
  almost conformally symplectic structure.

  Conversely, suppose that we are given a smooth line subbundle
  $\ell\subset\La^2T^*M$ in which each non--zero element of $\ell$ is
  non--degenerate. Then for each point $x\in M$, one defines $\Cal
  P_x$ to consist of all linear isomorphisms $V\to T_xM$ such that the
  induced isomorphism $\La^2V^*\to \La^2T^*M$ maps $b$ to an element
  of $\ell_x$. Of course, there is at least one such isomorphism and
  fixing this, one obtains a bijection between $\Cal P_x$ and
  $G=CSp(V)$. One then defines $\Cal P$ to be the disjoint union of
  the spaces $\Cal P_x$ and obtains a natural projection $\Cal P\to
  M$. It is then easy to see that this is a smooth principal
  $G$--bundle over $M$ which defines a reduction of structure group as
  required.

(2) We know that $\partial$ is $CSp(V)$ equivariant, and we analyze
  its kernel and its image as representations of $Sp(V)$. We have
  noted above that $\frak{csp}(V)\cong\Bbb R\oplus S^2V$. Tensorizing
  with $V^*\cong V$, it is well known that $S^2V\otimes V$ decomposes
  as a direct sum of three irreducible representations. Using $b$, one
  can form a trace $S^2V\otimes V\to V$ and the kernel of this trace
  decomposes as the direct sum of $S^3V$ and an irreducible
  representation $W$ (the kernel of the symmetrization map inside the
  tracefree part).

Likewise, $\La^2V^*$ decomposes as $\La^2_0V^*\oplus\Bbb R$. Forming
the tensor product with $V\cong V^*$, $\La^2_0V^*\otimes V^*$ again
decomposes as a direct sum of three irreducibles if $n>4$, while there
are only two irreducible components for $n=4$. Using $b^{-1}$, one
defines a trace $\La^2_0V^*\otimes V^*\to V^*$, and its kernel
decomposes into $\La^3_0V^*\oplus W'$, where the first summand is
non--zero only for $n>4$ and $W'$ is the kernel of the alternation
within the tracefree part. It is well known that $W$ and $W'$ are
isomorphic irreducible representations of $Sp(V)$.

Since the irreducible representation $S^3V$ does not occur in
$\La^2V^*\otimes V$, it must be contained in
$\frak{csp}(V)^{(1)}$. One then easily verifies directly that
$\partial$ is injective on the sum of two copies of $V$ contained in
$V^*\otimes \frak{csp}(V)$ and non--zero (and thus injective) on the
irreducible subspace $W$. This proves the claim on the first
prolongation. Moreover, it also implies that $\im(\partial)$ is
isomorphic to $W\oplus V\oplus V$, which immediately implies the
second claim.

(3) For an almost conformally symplectic structure
$\ell\subset\La^2T^*M$, consider the corresponding first order
structure $\Cal P\to M$ with structure group $G$. By construction, a
linear connection $\nabla$ on $TM$ is induced by a principal
connection $\ga$ on $\Cal P$ if and only if the line subbundle $\ell$
is preserved by the induced connection on $\La^2T^*M$. More
explicitly, for a vector field $\xi\in\frak X(M)$ and a local
non--vanishing section $\tau$ of $\ell$, also $\nabla_\xi\tau$ has to be
a section of $\ell$.

For such a section $\tau$, we define $i_T\tau$ to be the map
$(\xi,\eta,\zeta)\mapsto\tau(T(\xi,\eta),\zeta)$. From part (2) we
conclude that the intrinsic torsion of our geometry vanishes if and
only if the three form obtained by alternating $i_T\tau$ is a section
of the trace part of $\La^3T^*M$, i.e.~can be written as
$\tau\wedge\ph$ for some $\ph\in\Om^1(M)$. Now it is well known that
for a torsion--free connection, $d\tau$ can be computed as the
alternation of $\nabla\tau$. In the presence of torsion, this formula has
to be modified by adding a non--zero multiple of the alternation of
$i_T\tau$. But since the alternation of $\nabla\tau$ automatically is
a section of the trace part, we conclude that vanishing of the
intrinsic torsion is equivalent to the fact that $d\tau$ is a section
of the trace part. So this means $d\tau=\tau\wedge\ph$ for some
$\ph\in\Om^1(M)$ and from Proposition \ref{prop2.1} we know that, for
$n>4$, this is equivalent to $\ell$ defining a conformally symplectic
structure.

(4) From the proof of part (3) we know that for a locally
non--vanishing section $\al$ of $\ell$ and any connection $\nabla$
compatible with the structure, we have $\nabla\al=\ph\otimes\al$ for
some one--form $\ph$ on $M$. If $\nabla$ is torsion free, then the
alternation $\ph\wedge\al$ of this coincides with $d\al$. So if
$d\al=0$, then $0=\ph\wedge\al$ and we have noted above that by
non--degeneracy of $\al$ this implies $\ph=0$. Hence any locally
closed section of $\ell$ is parallel for $\nabla$, which implies both
claims.  
\end{proof}

\section{PACS-- and PCS--structures}\label{3}
In this section, we introduce the geometric structures studied in this
article and describes their basic properties. The theory of simple Lie
algebras leads to a family of subalgebras of conformally symplectic
Lie algebras via so--called contact gradings. Groups with this Lie
algebra then give rise to geometric structures, each of which has an
underlying almost conformally symplectic structure. Via the
classification of simple Lie algebras, one can describe the resulting
geometric structures explicitly.

\subsection{Contact gradings of simple Lie algebras}\label{3.1}
Let $\frak g$ be a simple Lie algebra over $\Bbb K=\Bbb R$ or $\Bbb
C$. A \textit{contact grading} on $\frak g$ is a decomposition 
$$
\frak g=\frak g_{-2}\oplus\frak g_{-1}\oplus\frak g_0\oplus\frak
g_1\oplus\frak g_2, 
$$ such that $[\frak g_i,\frak g_j]\subset\frak g_{i+j}$ (with $\frak
g_{\ell}=\{0\}$ if $|\ell|>2$), and such that $\dim_{\Bbb K}(\frak
g_{-2})=1$ and the Lie bracket of $\frak g$ restricts to a
non--degenerate bilinear map $\frak g_{-1}\x\frak g_{-1}\to\frak
g_{-2}$. Since $\frak g_{-2}$ is one--dimensional, this bracket gives
rise to a well defined one--dimensional subspace in $\La^2(\frak
g_{-1})^*$, in which each non--zero element is non--degenerate.

Gradings of this kind are used in the theory of quaternionic and
para--qua\-ter\-nio\-nic symmetric spaces and in the theory of
parabolic contact structures, the latter will be discussed briefly in
the second part of this series. It is well known that any complex
simple Lie algebra admits (up to isomorphism) a unique grading of this
type. Moreover, this grading can be restricted to almost all
non--compact real forms, see Section 3.2.4 and Example 3.2.10 of
\cite{book} for a complete discussion. 

From the grading property, it is clear that $\frak g_0$ acts on each
$\frak g_i$ via the restriction of the adjoint action. It is well
known that the resulting representation $\frak g_0\to\frak{gl}(\frak
g_{-1})$ is faithful. Since the adjoint action preserves the Lie
bracket, we conclude that it preserves the line in $\La^2(\frak
g_{-1})^*$ constructed above, so we actually get an inclusion $\frak
g_0\hookrightarrow\frak{csp}(\frak g_{-1})$. 

Next, we have to choose a group $G_0$ with Lie algebra $\frak
g_0$. For the general discussion of PACS--structures, we only have to
assume that the representation $\frak g_0\to\frak{csp}(\frak g_{-1})$
integrates to a representation $G_0\to CSp(\frak g_{-1})$, which then
is infinitesimally faithful by construction. As usual in the theory of
first order geometric structures, different choices of groups leads to
very similar geometries. We will describe the geometries for one
natural choice of group below and not go into further details on
possible other choices of groups. When dealing with contactifications
in the second part of this series of articles, we will have to
restrict the choice of groups a bit.

Having chosen $G_0$, we can consider first order $G_0$--structures on
smooth manifolds of dimension $\dim(\frak g_{-1})$, and by
construction any such structure will have an underlying almost
conformally symplectic structure.  There is a particular case here,
that we have to exclude from the further discussion. Namely, if one
takes $\frak g$ of type $C_n$, i.e.~if $\frak g$ is a symplectic Lie
algebra, then the resulting contact grading has the property that
$\frak g_0$ is the full Lie algebra $\frak{csp}(\frak g_{-1})$. Hence
in this case, a reduction to a structure group with Lie algebra $\frak
g_0$ is essentially only an almost conformally symplectic structure,
and there is no additional structure there. Therefore, in the rest of
this article, we will always assume that $\frak g$ is not of type
$C_n$. 

\begin{definition}\label{def3.1}
Let $\frak g$ be a simple Lie algebra over $\Bbb R$ which admits a
contact grading, let $\frak g_0\hookrightarrow\frak{csp}(\frak
g_{-1})$ be the corresponding representation. Let $G_0$ be a Lie group
with Lie algebra $\frak g_0$ such that there is a corresponding
homomorphism $G_0\to CSp(\frak g_{-1})$.

(1) The first order structures with structure group $G_0$ on manifolds
of dimension $\dim(\frak g_{-1})$ coming from the representation
$G_0\to CSp(\frak g_{-1})$ are called \textit{parabolic almost
  conformally symplectic structures} or PACS--structures (associated to
$G_0$).

(2) A \textit{parabolic conformally symplectic structure} or
PCS--structure is a PACS--structure for which the underlying almost
conformally symplectic structure is conformally symplectic.
\end{definition}

\begin{remark}\label{rem3.1}
(i) We will refine the terminology on PACS-- and PCS--structures in
  the discussion of the individual examples (according to the
  classification of simple Lie algebras admitting a contact grading)
  in the rest of this section. 

(ii) The Lie subalgebra $\frak g_0$ is known to be reductive, with
  center of dimension $2$ in the $A_n$--case and dimension $1$ in all
  other cases. There is a natural codimension--one subalgebra
  $\frak g^0_0\subset\frak g_0$, consisting of those elements of
  $\frak g_0$ which act trivially on $\frak g_{\pm 2}$ under the
  adjoint action. By construction, this coincides with $\frak
  g_0\cap\frak{sp}(\frak g_{-1})$, and, apart from the $A_n$--case, it
  also coincides with the semisimple part of $\frak g_0$. The
  resulting subalgebras of symplectic Lie algebras are exactly the
  \textit{special symplectic subalgebras} as introduced by M.~Cahen
  and L.~Schwachh\"ofer in \cite{Cahen-Schwachhoefer}. 

The main notion introduced in that reference is the one of a
\textit{special symplectic connection}, a torsion free connection,
which preserves a symplectic form and whose curvature lies in a
certain space $\Cal R_{\frak g_0^0}$ associated to a special
symplectic subalgebra $\frak g_0^0$. The family of special symplectic
connections in particular includes all affine connections having
exceptional holonomy of symplectic type. A detailed description of
$\Cal R_{\frak g_0^0}$ will be given in Section \ref{4.7}, where we
also discuss the relation between PCS--structures and special
symplectic connections.
\end{remark}

\subsection{PACS--structures of K\"ahler and para--K\"ahler
  type}\label{3.2} 
To obtain more explicit descriptions of PACS--structures, we have to
go through the list of contact gradings of simple Lie algebras. From
their use in the theory of parabolic contact structures, algebraic
descriptions of the resulting subalgebras in conformally symplectic
algebras are available. This can be directly converted into
information on the corresponding first order structure. 

We start this discussion with simple Lie algebras of type $A_n$,
i.e.~real forms of $\frak{sl}(n,\Bbb C)$. It is well known that there
are three different types of real forms of this algebra, namely
$\frak{sl}(n,\Bbb R)$, $\frak{sl}(n/2,\Bbb H)$ (for even $n$) and
$\frak{su}(p,q)$ with $p+q=n$. Contact gradings are available on
$\frak{sl}(n,\Bbb R)$ and on $\frak{su}(p,q)$ if both $p$ and $q$ are
non--zero. 

The contact grading of $\frak{sl}(n+2,\Bbb R)$ is described in Section
4.2.3 of \cite{book}. The grading component $\frak g_{-1}\cong\Bbb
R^{2n}$ splits as a direct sum $\frak g_{-1}^E\oplus\frak g_{-1}^F$ of
two subspaces of dimension $n$, which both are isotropic for the Lie
bracket. The subalgebra $\frak g_0\subset\frak{csp}(\frak g_{-1})$
consists of those maps which preserve this decomposition. A natural
choice for a group $G_0$ thus is the subgroup of $CSp(\frak g_{-1})$
consisting of all maps which preserve that decomposition of $\frak
g_{-1}$. Motivated by the description in the following Proposition, we
will refer to the corresponding geometric structure as a
\textit{PACS--structure of para--K\"ahler type}.

The contact grading of $\frak{su}(p+1,q+1)$ for $p+q=n$ is discussed
in Section 4.2.4 of \cite{book}. Here $\frak g_{-1}$ is a complex
vector space and the bracket has the property that $[iX,iY]=[X,Y]$ for
all $X,Y\in\frak g_{-1}$. Moreover, $\frak g_0$ consists of all maps
in $\frak{csp}(\frak g_{-1})$ which are complex linear. This implies
that $[\ ,\ ]$ is the imaginary part of a $\frak g_{-2}$--valued
Hermitian form, which turns out to have signature $(p,q)$. A natural
choice for a group $G_0$ thus is the subgroup of complex linear maps
in $CSp(\frak g_{-1})$, so $G_0\cong CU(p,q)$ is a conformal
pseudo--unitary group. We will call the corresponding geometric
structure a \textit{PACS--structure of K\"ahler type}.

To formulate the description of these types of PACS--structures, let
us recall some concepts. An \textit{almost Hermitian metric} on a
smooth manifold of even dimension $2n$ is given by an almost complex
structure $J$ on $M$ and a pseudo--Riemannian metric $g$ which is
Hermitian with respect to $J$, i.e.~such that
$g(J\xi,J\eta)=g(\xi,\eta)$ for all tangent vectors $\xi$ and
$\eta$. Given such a structure, any conformal rescaling of $g$ defines
an almost Hermitian metric on $M$, too, so it is no problem to talk
about a conformal class of almost Hermitian metrics. The basic
properties of almost Hermitian metrics are analyzed in the seminal
article \cite{G-H} of Gray and Hervella. The main ingredient there is
the \textit{fundamental two--form} of $g$ defined by
$\om(\xi,\eta):=-g(\xi,J\eta)$, so this is non--degenerate in each
point. In particular, $(g,J)$ is called an \textit{almost K\"ahler
  metric} if and only if $\om$ is closed.

Similarly, an \textit{almost para--Hermitian metric} on a smooth
manifold $M$ of even dimension $2n$ is defined as a decomposition
$TM=E\oplus F$ into two sub--bundles of rank $n$ and a
pseudo--Riemannian metric $g$ on $M$, for which both $E$ and $F$ are
isotropic. This implies that $g$ has split--signature $(n,n)$. As
above, it is no problem to consider conformal classes of almost
para--Hermitian metrics. The similarity to the Hermitian case becomes
evident if one describes the decomposition as an \textit{almost
  para--complex structure} $\Cal J:TM\to TM$, where $\Cal J$ acts as
the identity on $E$ and as minus the identity on $F$, so $\Cal
J^2=\id$. Then the compatibility of $g$ with the decomposition is
equivalent to the fact that $g(\Cal J\xi,\eta)=-g(\xi,\Cal J\eta)$, and
one defines a fundamental two--form $\om$ as above. The metric is then
called \textit{almost para--K\"ahler} if and only if $\om$ is closed.

\begin{prop}\label{prop3.2}
(1) A PACS--structure of K\"ahler type of signature $(p,q)$ on a
  smooth manifold $M$ of real dimension $2(p+q)\geq 4$ is equivalent
  to a conformal class of almost Hermitian metrics $(g,J)$ on $M$. 

(2) A PACS--structure of para--K\"ahler type on a smooth manifold $M$
of real dimension $2n\geq 4$ is given by a conformal class of
para--Hermitian metrics $(g,TM=E\oplus F)$ on $M$. 

(3) In both cases, the underlying almost conformally symplectic
structure $\ell\subset\La^2T^*M$ is spanned by the fundamental
two--forms of the metrics in the class. In particular, the structure
is PCS if and only if the conformal class locally contains
almost (para--)K\"ahler metrics, which then are unique up to a
constant factor.
\end{prop}
\begin{proof}
(1) From the description of $G_0$ it is clear that a first order
  structure corresponding to this group on $M$ is equivalent to an
  almost conformally symplectic structure $\ell\subset\La^2T^*M$ and
  an almost complex structure $J$ on $M$ such that for each $x\in M$
  any element of $\ell_x$ is Hermitian for $J_x$. Now given a
  non--zero element $\om\in\ell_x$, one defines $g_x:T_xM\x
  T_xM\to\Bbb R$ by $g_x(\xi,\eta)=\om(\xi,J\eta)$, which by
  construction is Hermitian and of signature $(p,q)$. Different
  elements in $\ell_x$ lead to proportional metrics, so we obtain a
  conformal class of almost--Hermitian metrics. Conversely, given such
  a class, one takes the almost complex structure $J$ together with
  the line $\ell$ spanned by the fundamental two--forms to obtain a
  first order structure with group $G_0$. 

(2) is proved in the same way as (1). 

(3) From the construction in the proof of (1) it is clear that
  $\ell\subset\La^2T^*M$ is spanned by the fundamental two--forms of
  the metrics in the class. But then by definition, a local section of
  $\ell$ is closed if and only if the corresponding local metric is
  almost (para--)K\"ahler. 
\end{proof}

\subsection{PACS--structures of Grassmannian type}\label{3.3}
We next move to real forms of the complex orthogonal Lie algebras
$\frak{so}(n,\Bbb C)$, i.e.~to types $B_m$ and $D_m$ in the
classification of simple Lie algebras. In view of the isomorphisms
$B_2\cong C_2$ and $D_3\cong A_3$, we only have to look at case $n\geq
7$ here. The obvious real forms in this case are the orthogonal Lie
algebras $\frak{so}(p,q)$ with $p+q=n$, and these admit a contact
grading provided that $p,q\geq 2$. For the even orthogonal Lie
algebras, there is a second real form which is discussed in Section
\ref{3.4} below.

So we have to consider real Lie algebras of the form
$\frak{so}(p+2,q+2)$ with $p+q=n\geq 3$ and their contact gradings are
described in Section 4.2.5 of \cite{book}. In this case $\frak g_{-1}$
is a space of matrices, $\frak g_{-1}\cong M_{n,2}(\Bbb R)=\Bbb
R^{2*}\otimes\Bbb R^n$ where $n=p+q$. Correspondingly, there is a
natural inclusion of $\frak s(\frak{gl}(2,\Bbb
R)\oplus\frak{gl}(n,\Bbb R))$ into $\frak{gl}(\frak g_{-1})$ and $\frak
g_0$ is contained in that subalgebra and hence acts by maps preserving
the tensor product decomposition. Hence, as a representation of $\frak
g_0$, we get
$$ 
\La^2(\frak g_{-1})^*=\La^2(\Bbb R^2\otimes\Bbb R^{n*})\cong
(S^2\Bbb R^2\otimes\La^2\Bbb R^{n*})\oplus (\La^2\Bbb R^2\otimes
S^2\Bbb R^{n*}), 
$$ and the line corresponding to the bracket $\frak g_{-1}\x\frak
g_{-1}\to\frak g_{-2}$ is contained in the second summand. Hence it
gives rise to an inner product on $\Bbb R^n$ defined up to scale, and
this has signature $(p,q)$. Given this line, the algebra $\frak g_0$
consists of those elements of $\frak{csp}(\frak g_{-1})$ which are
compatible with the tensor product decomposition. 

Hence we obtain a natural choice of group $G_0$ by intersecting
$CSp(\frak g_{-1})$ with the subgroup of $GL(\frak g_{-1})$ consisting
of those maps which preserve the tensor product decomposition. The
latter group is the image of the natural homomorphism $GL(2,\Bbb R)\x
GL(n,\Bbb R)\to GL(\frak g_{-1})$ obtained by multiplying with
matrices from both sides. Up to a covering, $G_0$ is isomorphic to
$GL(2,\Bbb R)\x SO(p,q)$.

To describe the corresponding PACS--structures, let us recall a bit of
background. Suppose that $M$ is a manifold of real dimension
$2n$. Then an \textit{almost Grassmannian structure} of type $(2,n)$
on $M$ is given by two auxiliary vector bundles $E$ and $F$ over $M$
of rank two and $n$, respectively, together with isomorphisms $TM\cong
L(E,F)=E^*\otimes F$ and $\La^2E\cong\La^nF$. Equivalently, these can
be described as first order structures corresponding to subgroup of
$GL(M_{n,2}(\Bbb R))$ described above.

For $n=2$, an almost Grassmannian structure is equivalent to a split
signature conformal structure, and also for $n>2$, almost Grassmannian
structures are similar to conformal structures in several
respects. This is related to the fact that they can be equivalently
described by a canonical Cartan connection with homogeneous model the
Grassmannian of two--planes in $\Bbb R^{n+2}$ viewed as a homogeneous
space of $SL(n+2,\Bbb R)$, see Section 4.1.3 of \cite{book}. Hence
they fall into the class of so--called AHS--structures, a subclass of
parabolic geometries which also contains conformal structures.

Given an almost Grassmannian structure $TM\cong E^*\otimes F$ on $M$,
one obtains an isomorphism $\La^2T^*M\cong
(S^2E\otimes\La^2F^*)\oplus(\La^2E\otimes S^2F^*)$, so there are two
different types of two--forms. In particular, one calls a two--form
\textit{Hermitian} (in the Grassmannian sense) if it lies in
$\La^2E\otimes S^2F^*$. Since $\La^2E$ is a line bundle, a Hermitian
two--form induces a symmetric bilinear form on the bundle $F$
determined up to scale. In particular, in the non--degenerate case,
one can associate to such a two--form a well defined signature $(p,q)$
with $p+q=n$.

In view of this discussion, it is natural to call the PACS--structures
corresponding to the contact gradings of real orthogonal algebras
\textit{of Grassmannian type} and the following result is evident.

\begin{prop}\label{prop3.3}
  A PACS--structure of Grassmannian type on a smooth manifold $M$ of
  dimension $2(p+q)\geq 6$ is given by an almost Grassmannian
  structure $TM\cong E^*\otimes F$ of type $(2,p+q)$ on $M$ together
  with an almost conformally symplectic structure
  $\ell\subset\La^2T^*M$, which is Hermitian in the Grassmannian
  sense, i.e.~contained in the subbundle $\La^2E\otimes S^2F^*$, and
  has signature $(p,q)$.
\end{prop}

\begin{remark}\label{rem3.3}
For comparison with the quaternionic case to be discussed below, we
note that these almost Grassmannian structures admit an alternative
description. Given an almost Grassmannian structure $TM\cong
E^*\otimes F$, any endomorphism of $E$ induces an endomorphism of
$TM$. For each $x\in M$, the space $L(T_xM,T_xM)$ naturally is an
associative algebra under composition, and we obtain an inclusion of
$L(E_x,E_x)\cong M_2(\Bbb R)$ as a subalgebra. These clearly fit
together to define a bundle of subalgebras in the locally trivial
bundle $L(TM,TM)$ of (associative) algebras. 

Since the algebra $M_2(\Bbb R)$ is the unique 4--dimensional normed
real algebra with indefinite quadratic form (given by the
determinant), it is also called the algebra of \textit{split
  quaternions}. Conversely, given such a bundle of subalgebras, one
can use the fact that $M_2(\Bbb R)$ is a simple algebra to recover a
tensor product decomposition of $TM$. Hence almost Grassmannian
structures of type $(2,n)$ are also called \textit{almost split
  quaternionic structures}.
\end{remark}

\subsection{PACS--structures of quaternionic type}\label{3.4}
As mentioned in Section \ref{3.3}, there is another real form of the even
orthogonal Lie algebras $\frak{so}(2n,\Bbb C)$, which admits a contact
grading. This is usually denoted by $\frak{so}^*(2n)$ and is related
to skew--Hermitian forms on quaternionic vector spaces. We will
discuss this rather briefly here and more thoroughly in the second
part of this series in the context of contactification. 

Recall that the concept of Hermitian forms makes sense over the
quaternions. A (quaternionically) Hermitian form on a (right) vector
space $V$ over the skew--field $\Bbb H$ of quaternions is defined as a
real bilinear map $b:V\x V\to\Bbb H$ such that for $v,w\in V$ and
$q\in\Bbb H$, one has $b(vq,w)=\overline{q}b(v,w)$, $b(v,wq)=b(v,w)q$
and $b(w,v)=\overline{b(v,w)}$. Here the bar denotes the usual
conjugation on $\Bbb H$. Likewise, a (quaternionically)
\textit{skew--Hermitian form} on $V$ is defined a real bilinear map
$\om:V\x V\to\Bbb H$ such that $\om(vq,w)=\overline{q}\om(v,w)$,
$\om(v,wq)=\om(v,w)q$ and $\om(w,v)=-\overline{\om(v,w)}$. (The analog
of this concept also exists over the complex numbers, but it is not
studied there, since multiplication by $i$ defines an isomorphism
between Hermitian and skew--Hermitian forms in the complex case.)

Similarly as for quaternionically Hermitian forms, a skew--Hermitian
form on $V$ can be recovered from its real part, which now is a
skew--symmetric real valued bilinear form on $V$ which is preserved by
multiplication by all unit quaternions. Conversely, any such bilinear
form can be uniquely extended to a quaternionically skew--Hermitian
form.

It turns out that on any finite dimensional quaternionic vector space
there is a unique (up to quaternionically linear automorphisms)
non--degenerate quaternionically skew--Hermitian form. Hence we can
choose one such form $\om$ on $\Bbb H^n$ and define $SO^*(2n)$ as the
group of all quaternionically linear automorphisms $A$ of $\Bbb H^n$
such that $\om(Av,Aw)=\om(v,w)$ for all $v,w\in\Bbb H^n$. This clearly
is a closed subgroup of $GL(n,\Bbb H)$ and hence a Lie group, and we
denote by $\mathfrak{so}^*(2n)$ its Lie algebra. In the definition of
$SO^*(2n)$, one may equivalently replace $\om$ by its real part, which
shows that, $SO^*(2n)$ is the intersection in $GL(4n,\Bbb R)$ of
$GL(n,\Bbb H)$ with the stabilizer of the real part of $\om$, which is
isomorphic to $Sp(4n,\Bbb R)$. 

It is easy to verify explicitly that $\mathfrak{so}^*(2n)$ is a real
form of the complex orthogonal algebra $\mathfrak{so}(2n,\Bbb
C)$. While for $n=2,3,4$, this is isomorphic to real forms as
discussed in Section \ref{3.3} above, one obtains a genuinely different real
form for $n\geq 5$. The following description of the associated
PACS--structure is also valid for $n=2,3,4$. 

The situation here is closely parallel to the Grassmannian case
treated in Section \ref{3.3}. Starting from $\frak g=\frak{so}^*(2n+2)$, we
obtain $\frak g_{-1}\cong\Bbb H^n$ and the bracket $[\ ,\ ]:\frak
g_{-1}\x\frak g_{-1}\to\frak g_{-2}$ is the real part of a
quaternionically skew--Hermitian form. The Lie subalgebra $\frak
g_0\subset\frak{csp}(\frak g_{-1})$ is isomorphic to $\Bbb
H\oplus\frak{so}^*(\frak g_{-1})$ with the first factor acting via
quaternionic scalar multiplications (which are not quaternionically
linear since $\Bbb H$ is non--commutative). As a natural choice for
$G_0$ we can then use the subgroup $\Bbb H^*\cdot SO^*(\frak
g_{-1})\subset CSp(\frak g_{-1})$ with the first factor acting by
scalar multiplications.

Now it is well known that a reduction of structure group of a manifold
of dimension $4n$ to the subgroup $\Bbb H^*\cdot GL(n,\Bbb H)\subset
GL(4n,\Bbb R)$ which is generated by quaternionic scalar
multiplications and quaternionically linear automorphisms of $\Bbb
H^n$ is equivalent to an \textit{almost quaternionic structure}. An
almost quaternionic structure can be equivalently described as a
bundle of subalgebras in the bundle $L(TM,TM)$ of associative algebras
with modeling algebra $\Bbb H$. A more traditional equivalent
description is as a rank--three subbundle of $L(TM,TM)$ which can
locally be spanned by three almost complex structures $I$, $J$, and
$K$ which satisfy the quaternion relations. In view of this, we call
the resulting PACS--structure \textit{of quaternionic type}, and the
following result is evident.

\begin{prop}\label{prop3.4}
  A PACS--structure of quaternionic type on a smooth manifold $M$ of
  dimension $4n\geq 8$ is given by an almost quaternionic structure on
  $M$ together with an almost conformally symplectic structure
  $\ell\subset\La^2T^*M$, such that each non--zero element of $\ell$
  is the real part of a quaternionically skew--Hermitian form.
\end{prop}

\begin{example}\label{ex3.4}
  Let us give an example of a homogeneous torsion--free PCS--structure
  of quaternionic type. This example will be discussed in more detail
  in the context of contactification in the next part of this series.

  It is well known that the complex Grassmannian $M:=Gr(2,\Bbb
  C^{n+2})$, viewed as a homogeneous space of $SU(n+2)$ is a very
  remarkable example of a symmetric space. This is due to the fact
  that it carries an invariant complex structure $J$ and an invariant
  quaternionic structure $Q$, with $J$ not being contained in
  $Q$. Moreover, there is an invariant Riemannian metric $g$ on $M$
  which is K\"ahler with respect to $J$ and quaternion--K\"ahler with
  respect to $Q$.  From Proposition \ref{prop3.2} we know that the
  K\"ahler metric $g$ determines a PCS--structure of K\"ahler type on
  $M$. It turns out that $\omega$ defines a second PCS--structure on
  $M$.

  \begin{cor}\label{cor3.4}
    Let $M$ be the symmetric space $SU(n+2)/S(U(n)\x U(2))$, let $\Cal
    Q$ be the invariant quaternionic structure on $M$ and let $\omega$
    be the K\"ahler form of the invariant K\"ahler metric on $M$. Then
    $\omega$ defines a PCS--structure of quaternionic type on $(M,\Cal
    Q)$.
  \end{cor}
  \begin{proof}
    We use the description of invariant structures of homogeneous
    spaces as developed in Section 1.4 of \cite{book}. The symmetric
    decomposition of $\frak g=\frak{su}(n + 2)$ as $\frak h\oplus\frak
    m$ can be seen by representing matrices in block form with blocks
    of size $2$ and $n$ as $\left(\begin{smallmatrix} A & -X^* \\ X &
      B\end{smallmatrix}\right)$. Here $A\in\frak u(2)$ and $B\in\frak
      u(n)$ are such that $\tr(A)+\tr(B)=0$ and form the $\frak
      h$-–part, while $X$ is a complex $(n\x 2)$-–matrix, which forms
      the $\frak m$--part. The action of the subgroup $H\cong S(U(n)\x
      U(2))$ corresponding to $\frak h$ comes from multiplying $X$ by
      unitary matrices from both sides. Since this action is complex
      linear for the usual multiplication by $i$, this multiplication
      gives rise to an invariant almost complex structure $J$ on $M$.

      On the other hand, multiplication from the right by elements of
      $\frak{su}(2)$ defines a three--dimensional subspace $Q\subset
      L_{\Bbb C}(\frak m,\frak m)$, which is $H$--invariant. The
      induced action of $H$ on $Q$ is via the adjoint action of
      $SU(2)\subset H$ on $\frak{su}(2)\subset \frak h$. Hence $Q$
      gives rise to a three dimensional smooth subbundle $\Cal
      Q\subset L(TM,TM)$. It is well known that $\frak{su}(2)$ is
      isomorphic to the imaginary quaternions, so $\Cal Q$ defines
      invariant almost quaternionic structure on $M$.

      Now the standard Hermitian inner product on the complex vector
      space $\frak m$ can be written as $(X,Y)\mapsto\tr(X^*Y)$.  This
      is Hermitian with respect to multiplication by $i$, so its real
      part extends to an invariant Riemannian metric $g$ on $M$,
      which is Hermitian with respect to $J$. Since by construction
      both $J$ and $g$ are parallel for the canonical connection on
      $M$ , this defines an invariant K\"ahler structure whose
      K\"ahler form $\om$ is induced by the imaginary part of the
      Hermitian inner product.  

      But now the imaginary unit quaternions correspond to those
      elements $A\in\frak{su}(2)$ for which $A^*=A^{-1}=-A$ or
      equivalently $A^2=-\id$. But for such an element, the Hermitian
      inner product satisfies
      $(AX,AY)\mapsto\tr(X^*A^{-1}AY)=\tr(X^*Y)$. Hence the real part
      of this inner product is Hermitian in the quaternionic sense, so
      the metric g is quaternion–-K\"ahler with respect to
      $\Cal Q$. Likewise, the imaginary part is Hermitian in the
      quaternionic sense, so $\om$ defines a PCS-–structure of
      quaternionic type with respect to $\Cal Q$.
  \end{proof}
\end{example}

\subsection{Exceptional PACS--structures}\label{3.5}
The contact gradings of real forms of the exceptional Lie algebras are
discussed in Section 4.2.8 of \cite{book}. In each case, one obtains a
reductive Lie algebra $\frak g_0$ with one--dimensional center
together with an irreducible representation on $\frak g_{-1}$. All
these representations have the property that their second exterior
square contains a unique one--dimensional invariant subspace (and in
fact only one more irreducible component). Hence any reductions of
structure group of a manifold of dimension $\dim(\frak g_{-1})$
automatically gives rise to an almost conformally symplectic structure
(corresponding to the one--dimensional summand). In the exceptional
cases it thus seems less appropriate to us to view a PACS--structure
as an almost conformally symplectic structure plus some additional
data. Rather than that one should view this as first order structures
which carry an underlying almost conformally symplectic structure.
The resulting algebras and representations are collected in the
following table.

\medskip

\begin{tabular}{|l|l|l|l|l|}
\hline type & $\frak g_0$ & $\frak g_{-1}$ & $\dim(\frak
g_{-1})$ & further real forms \\
\hline
$G_2$ & $\frak{gl}(2,\Bbb R)$ & $S^3\Bbb R^2$ & 4 & --\\
\hline
$F_4$ & $\frak{csp}(6,\Bbb R)$ & $\La^3_0\Bbb R^6$ & 14 & -- \\
\hline
$E_6$ & $\frak{gl}(6,\Bbb R)$  & $\La^3\Bbb R^6$ & 20 &
$\frak{csu}(3,3), \frak{csu(5,1)}$ \\
\hline
$E_7$ & $\frak{cso}(6,6)$ & Spin & 32 & $\frak{cso}(4,8)$,
$\frak{cso}^*(12)$ \\ 
\hline
$E_8$ & split $\frak e_7$ &  & 56 & one more \\
\hline 
\end{tabular}

\medskip

For the $F_4$--case, one has to view $\Bbb R^6$ as a symplectic vector
space and then $\La^3_0\Bbb R^6$ denotes the tracefree part in the
third exterior power. For the $F_4$-- and $E_6$--cases, the underlying
almost conformally symplectic structure comes from the wedge product
$\La^3\Bbb R^6\x\La^3\Bbb R^6\to\La^6\Bbb R^6$. For the two
conformally unitary algebras showing up in the $E_6$--case, one has to
use appropriate real subrepresentations in $\La^3\Bbb C^6$, and
likewise in the $E_7$--case, one needs real spin representations
(which restricts the available signatures). To our knowledge, there is
no established name for the $56$--dimensional representation of $\frak
e_7$, it is the irreducible representation of lowest dimension of this
algebra. The two real forms which show up here are the ones for which
the complex representation of dimension $56$ admits a real form.

In the cases associated to $G_2$, $F_4$ and $E_6$, one can describe
the reductions of structure group more explicitly. For the
PACS--structure determined by $G_2$ one has a manifold of dimension
$4$, together with an identification $TM\cong S^3E$ for an auxiliary
rank--two bundle $E$. Likewise the PACS--structures corresponding to
$F_4$ and $E_6$ can be described in terms of an auxiliary bundle $E$
of rank $6$ (endowed with a symplectic form in the $F_4$--case). For
the remaining two cases, finding a similar description would first
require a characterization of the $\frak g_0$--representations in
question.

\section{Canonical connections}\label{4}

We will next analyze the PACS--structures introduced in Section
\ref{3} from the point of view of compatible connections on
G--structures as discussed in Section \ref{2.2}. We will prove that
for all the subalgebras $\frak g_0\subset\frak{gl}(\frak g_{-1})$
coming from contact gradings of simple Lie algebras of type different
from $C_n$, the first prolongation (in the sense of Sternberg as
introduced in Section \ref{2.2}) vanishes. Moreover, there is a
natural choice of a normalization condition, so that any
PACS--structure gives rise to a canonical linear connection on the
tangent bundle. The form of the resulting intrinsic torsion can be
described explicitly for each of the structures. The essential tool
for proving all these results is Kostant's theorem (see \cite{Kostant}
or the discussion in Section 3.3 of \cite{book}) and its applications
to parabolic contact structures.

\subsection{Kostant's Theorem}\label{4.1}
Let $\frak g$ be a simple Lie algebra endowed with a contact grading
$\frak g=\frak g_{-2}\oplus\frak g_{-1}\oplus\frak g_0\oplus\frak
g_1\oplus\frak g_2$ as introduced in Section\ref{3.1}. Then $\frak g_-:=\frak
g_{-2}\oplus\frak g_{-1}$ is a nilpotent subalgebra of $\frak g$ (a
Heisenberg algebra), so $\frak g$ is naturally a module over $\frak
g_{-}$. Then there is the standard complex for computing the Lie
algebra cohomology spaces $H^*(\frak g_-,\frak g)$, consisting of the
chain spaces $C^k(\frak g_-,\frak g):=\La^k(\frak g_-)^*\otimes\frak
g$ and differentials $\partial_K:C^k(\frak g_-,\frak g)\to
C^{k+1}(\frak g_-,\frak g)$. Viewing $C^k(\frak g_-,\frak g)$ as the
space of alternating multilinear maps, this differential is given by
\begin{equation*}
\begin{aligned}
\partial_K\ph(X_0&,\dots,X_k):=\sum_{i=0}^k(-1)^i[X_i, 
\ph(X_0,\dots,\widehat{X_i},\dots,X_k)]\\
&+\sum_{i<j}(-1)^{i+j}
\ph([X_i,X_j],X_0,\dots,\widehat{X_i},\dots,\widehat{X_j},\dots,X_k), 
\end{aligned}
\end{equation*}
where the Lie brackets are in $\frak g$.  (The subscript is chosen to
distinguish this ``Kostant--differential'' from the Spencer
differential introduced in Section \ref{2.2}, which we will denote by
$\partial_S$ in what follows.)

Observe that there is a natural notion of homogeneity for multilinear
maps mapping $\frak g_{-}$ to $\frak g$, and $\partial_K$ preserves
homogeneities (since the Lie brackets do). Consequently, one can split
the whole complex $(C^*(\frak g_-,\frak g),\partial_K)$ and thus also
its cohomology according to homogeneity. Moreover, the grading
property implies that the restriction of the adjoint action to $\frak
g_0$ preserves each grading component $\frak g_i$. Hence all the
spaces $C^k(\frak g_-,\frak g)$ are $\frak g_0$--modules and one
easily verifies that $\partial_K$ is $\frak g_0$--equivariant. Thus,
each of the cohomology spaces naturally is a representation of $\frak
g_0$, and this structure is crucial for the description of the
cohomology given in Kostant's theorem.

The first step towards proving Kostant's theorem is the construction
of an adjoint $\partial^*:C^k(\frak g_-,\frak g)\to C^{k-1}(\frak
g_-,\frak g)$ to $\partial_K$, usually called the \textit{Kostant
  codifferential}. Using this, one defines the \textit{Kostant
  Laplacian} $\square:=\partial^*\partial_K+\partial_K\partial^*$
which maps $C^k(\frak g_-,\frak g)$ to itself. The adjointness of
$\partial^*$ and $\partial_K$ leads to an algebraic Hodge
decomposition, which in particular implies that $H^k(\frak g_-,\frak
g)\cong\ker(\square)\subset C^k(\frak g_-,\frak g)$ as a $\frak
g_0$--module and elements in $\ker(\square)$ are referred to as
``harmonic''. Now Kostant's theorem describes $\ker(\square)$ as a
representation of $\frak g_0$. It turns out that this representations
is simply reducible, i.e.~a direct sum of pairwise non--isomorphic
irreducible representations. The highest weights of each of these
irreducible components as well as the cohomology degree in which it is
contained can be computed in terms of the Weyl group of $\frak
g$. This computation is completely algorithmic. We will not discuss
the general result but just use the available information on the
cohomology groups as it is needed.

\subsection{Maximality}\label{4.2}
The first result we deduce is at least implicitly in
\cite{Cahen-Schwachhoefer} and it is based on the description of
$H^1(\frak g_-,\frak g)$ obtained from Kostant's theorem. 

\begin{prop}\label{prop4.2}
Suppose that $\frak g$ is not of type $A_n$ or $C_n$. Then the
subalgebra $\frak g_0\subset\frak{csp}(\frak g_{-1})$ coming from the
contact grading of $\frak g$ is a maximal subalgebra. Indeed, any
$\frak g_0$--invariant subspace of $\frak{csp}(\frak g_{-1})$ strictly
containing $\frak g_0$ coincides with $\frak{csp}(\frak g_{-1})$.
\end{prop}
\begin{proof}
If $\frak g$ is not of type $A_n$, then the parabolic subalgebra
$\frak g_0\oplus\frak g_1\oplus\frak g_2$ is well known to be
maximal (see Section 3.4.2 of \cite{book}). This implies that there is
only one simple reflection in the corresponding Hasse--diagram, so by
Kostant's theorem $H^1(\frak g_-,\frak g)$ is an irreducible
representation. Moreover, if $\frak g$ is not of type $C_n$, then
Propositions 3.3.7 and 4.3.1 of \cite{book} imply that this
irreducible representation must be contained in homogeneity zero. 

So let us assume that $\ph:\frak g_-\to\frak g$ is a linear map which
is homogeneous of degree zero, i.e.~consists of components
$\ph_i:\frak g_i\to\frak g_i$ for $i=-1,-2$. By definition,
$\partial_K\ph=0$ is equivalent to
$\ph_{-2}([X,Y])=[\ph_{-1}(X),Y]-[Y,\ph_{-1}(X)]$. But this exactly
means that $\ph_{-1}\in\frak{csp}(\frak g_{-1})$ and then it uniquely
determines $\ph_{-2}$. On the other hand, the homogeneity--zero part
of $\partial_K:\frak g\to C^1(\frak g_-,\frak g)$ is just
$\operatorname{ad}:\frak g_0\to L(\frak g_-,\frak g_-)$. Thus we
conclude that $H^1(\frak g_-,\frak g)\cong \frak{csp}(\frak
g_{-1})/\frak g_0$ as a $\frak g_0$--module. Since any $\frak
g_0$--invariant subspace in $\frak{csp}(\frak g_{-1})$ strictly
containing $\frak g_0$ descends to a non--zero $\frak g_0$--invariant
subspace of $\frak{csp}(\frak g_{-1})/\frak g_0$, irreducibility of
$H^1(\frak g_-,\frak g)$ implies the claim.
\end{proof}

The main point about this result is that it adds to the perspective on
PACS--structures. Since an almost conformally symplectic structure
certainly does not determine a unique linear connection on the tangent
bundle, it tells us that there are no intermediate structures between
an almost conformally symplectic structure and a PACS--structure,
which could already determine a canonical connection.

From that point of view, it is also interesting to see what happens in
the $A_n$--case. Here the result depends on the real form under
consideration. For $\frak{su}(p+1,q+1)$ it turns out that $H^1(\frak
g_-,\frak g)$ is also irreducible (deducing this from Kostant's
theorem is slightly more complicated, since one has to look at the
complexification for which the cohomology splits into two
irreducibles). Hence in this case, one again obtains a maximal
subalgebra.

On the other hand, for the real form $\frak{sl}(n+2,\Bbb R)$ the
subalgebra $\frak g_0$ is really non--maximal. As we have seen in
Section \ref{3.2}, $\frak g_0\subset\frak{csp}(\frak g_{-1})$ consists of
those maps which preserve a decomposition of $\frak g_{-1}$ into a
direct sum of two Lagrangean subspaces. Now there are two intermediate
subalgebras, which consist of the maps preserving just one of these
two subspaces. The geometric structure these subalgebras correspond to
is of course a conformally symplectic structure together with a
distinguished Lagrangean distribution. But such a structure can
certainly not determine a canonical linear connection on the tangent
bundle of the manifold. Starting from any smooth manifold $N$, the
cotangent bundle $M:=T^*N$ carries such a structure coming from the
canonical symplectic form and the vertical distribution. Now any
diffeomorphism of $N$ induces a diffeomorphism of $M$, which preserves
both the canonical symplectic structure and the vertical subbundle.
Of course, the infinite dimensional group of diffeomorphisms of $N$
cannot preserve a linear connection on $TM$.

\subsection{Vanishing of the first prolongation}\label{4.3} 
To proceed towards the main result of this section, we need a few more
facts related to Kostant's theorem. For this step, the main tool is
the homogeneity--one--component of $H^2(\frak g_{-},\frak g)$. For
$C^2(\frak g_-,\frak g)$, the part of homogeneity one is the direct
sum of the spaces $\La^2\frak g_{-1}^*\otimes\frak g_{-1}$ and $\frak
g_{-2}^*\otimes\frak g_{-1}^*\otimes\frak g_{-2}$. On the one hand,
Kostant's theorem implies that the kernel of $\square$ in homogeneity
one is contained in $\La^2\frak g_{-1}^*\otimes\frak g_{-1}$, see
Lemma 4.2.2 in \cite{book}. On the other hand, the Lie bracket
identifies $\frak g_{-1}$ with $\frak g_{-1}^*\otimes\frak g_{-2}$,
and hence we can naturally view $\La^3\frak g_{-1}^*\otimes\frak
g_{-2}$ as a subspace in $\La^2\frak g_{-1}^*\otimes\frak g_{-1}$. The
same then applies to the tracefree part, and using this, we can now
formulate

\begin{thm}\label{thm4.3}
Let $\frak g$ be a simple Lie algebra, which is not of type $C_n$ and
admits a contact grading, and let $\frak g_0\subset\frak{csp}(\frak
g_{-1})\subset\frak{gl}(\frak g_{-1})$ be the corresponding
inclusion. 

Then the Spencer differential $\partial_S:\frak g_{-1}^*\otimes\frak
g_0\to \La^2\frak g_{-1}^*\otimes\frak g_{-1}$ is
injective. Moreover, the subspaces $\ker(\square)$ and $\La^3_0\frak
g_{-1}^*\otimes\frak g_{-2}$ of $\La^2\frak g_{-1}^*\otimes\frak
g_{-1}$ have zero intersection and their direct sum is a linear
complement to $\Im(\partial_S)$. 
\end{thm}
\begin{proof}
The homogeneity--one part of $(C^*(\frak g_-,\frak
g),\partial_K)$ can be decomposed as
$$
\xymatrix@R=8mm@C=8mm{%
\frak g_1 \ar[r]\ar[dr]_\al & \frak g_{-1}^*\otimes \frak
g_0\ar[r]^{\partial_S}\ar[dr] &\La^2\frak g_{-1}^*\otimes
\frak g_{-1}\ar[r]^\ga &\La^3\frak
g_{-1}^*\otimes\frak g_{-2} \\
& \frak g_{-2}^*\otimes\frak g_{-1} \ar[r]^{\be} \ar[ur] &\frak
g_{-1}^*\otimes\frak g_{-2}^*\otimes\frak g_{-2}\ar[ur]_i &}.
$$ 

Here the first column corresponds to $C^0$, the sum of the two spaces
in the next column is $C^1$, and so on. Moreover, we have split
$\partial_K$ into maps between the individual direct summands. The map
$\al$ by definition maps $Z\in\frak g_1$ to $\ad_Z:\frak
g_{-2}\to\frak g_{-1}$, which easily implies that $\al$ is a linear
isomorphism. The map $\be$ is just a tensor product of the linear
isomorphism $\frak g_{-1}\cong\frak g_{-1}^*\otimes\frak g_{-2}$
(defined by the bracket) with an identity map, so it is a linear
isomorphism, too. Next, $\frak g_{-1}^*\otimes\frak g_{-2}^*$
naturally includes as the trace part into $\La^3\frak g_{-1}^*$ and
$i$ is just the tensor product of this inclusion with the identity map
and thus is injective. Finally, from the definition of $\partial_K$,
one immediately concludes that $\ga$ is the composition of the
alternation with the obvious isomorphism $\La^2\frak
g_{-1}^*\otimes\frak g_{-1}\cong \La^2\frak g_{-1}^*\otimes\frak
g_{-1}^*\otimes\frak g_{-2}$.

Now suppose that $\ph\in\frak g_{-1}^*\otimes\frak g_0$ satisfies
$\partial_S(\ph)=0$. Viewing $\ph$ as an element of $C^1(\frak
g_{-},\frak g)$, we have $\partial_K(\ph)\in \frak
g_{-1}^*\otimes\frak g_{-2}^*\otimes\frak g_{-2}$ and hence
$0=\partial_K\partial_K\ph=i(\partial_K\ph)$. Since $i$ is injective,
we have $\partial_K\ph=0$. As we have noted in Section \ref{4.1}
already, Kostant's theorem implies that $H^1(\frak g_-,\frak g)$ is
concentrated in homogeneity zero, so there must be an element
$Z\in\frak g_1$ such that $\ph=\partial_K(Z)$. But this implies that
$\al(Z)=0$ and hence $Z=0$ and thus $\ph=0$, so injectivity of
$\partial_S$ follows.

Next consider $\ker(\square)\subset\La^2\frak g_{-1}^*\otimes\frak
g_{-1}$. By Kostant's theorem, each irreducible component of
$\ker(\square)$ occurs with multiplicity one in $C^*(\frak g_-,\frak
g)$. Hence we conclude that each such component must have zero
intersection with $\Im(\partial_S)$, which is a sum of irreducible
representations contained in $C^1(\frak g_{-},\frak g)$. Likewise, it
has to have zero intersection with the subspace $\La^3\frak
g_{-1}^*\otimes\frak g_{-2}$, which is a sum of irreducible
representations contained in $C^3(\frak g_{-},\frak g)$.

Now consider $\ph\in\frak g_{-1}^*\otimes\frak g_0$. Then
$0=\partial_K\partial_K\ph=\ga(\partial_S\ph)+i(\partial_K\ph-\partial_S\ph)$,
which implies that $\ga(\partial_S\ph)$ lies in the trace part of
$\La^3\frak g_{-1}^*\otimes\frak g_{-2}$. This shows that
$\Im(\partial_S)$ has zero intersection with $\La^3_0\frak
g_{-1}^*\otimes\frak g_{-2}$ and hence also with the direct sum of
that space and $\ker(\square)$.

To complete the proof, it thus suffices to show that the three
subspaces $\Im(\partial_S)$, $\ker(\square)$, and $\La^3_0\frak
g_{-1}^*\otimes\frak g_{-2}$ span $\La^2\frak g_{-1}^*\otimes\frak
g_{-1}$. Given $\ps\in \La^2\frak g_{-1}^*\otimes\frak g_{-1}$, there
is a unique element $\ps_3\in\La^3_0\frak g_{-1}^*\otimes\frak
g_{-2}\subset \La^2\frak g_{-1}^*\otimes\frak g_{-1}$ such that
$\ga(\ps_3)$ coincides with the tracefree part of $\ga(\ps)$. Hence
$\ga(\ps-\ps_3)$ is pure trace, so there is a unique element
$\tilde\ps\in\frak g_{-1}^*\otimes\frak g_{-2}^*\otimes\frak g_{-2}$
such that $i(\tilde\ps)=\ga(\ps-\ps_3)$ and hence
$\partial_K(\ps-\tilde\ps-\ps_3)=0$. Now Kostant's theorem implies
that $\ker(\square)$ is a linear complement to $\Im(\partial_K)$ in
$\ker(\partial_K)$. Hence there are elements $\ps_2\in\ker(\square)$
and $\tilde\ph\in C^1(\frak g_-,\frak g)$ such that
$\ps-\tilde\ps-\ps_3=\ps_2+\partial_K\tilde\ph$. Finally, there is an
element $Z\in\frak g_1$ such that $\al(Z)$ coincides with the
component of $\tilde\ph$ in $\frak g_{-2}^*\otimes\frak g_{-1}$, and
we put $\ph=\tilde\ph-\partial_KZ\in\frak g_{-1}^*\otimes\frak
g_0$. Then by construction, we have
$\partial_K\ph=\partial_K\tilde\ph$, so we conclude that
$\ps-\tilde\ps-\ps_3=\ps_2+\partial_K\ph$. Now the component of
$\partial_K\ph$ in $\La^2\frak g_{-1}^*\otimes\frak g_{-1}$ equals
$\partial_S\ph$ by construction, so looking at the components in that
space, we get $\ps=\partial_S\ph+\ps_2+\ps_3$ and the proof is
complete.
\end{proof}

\subsection{Canonical connections for PACS-structures}\label{4.4}
Converting the algebraic results of Theorem \ref{thm4.3} to geometry now
is an easy task. Consider a contact grading on $\frak g$ (which is not
of type $C_n$) and a corresponding group $G_0$. Then by definition a
PACS--structure of type $G_0$ on a smooth manifold $M$ is
given by a reduction of structure group of the linear frame bundle to
$G_0$. We denote by $\Cal G_0\to M$ the corresponding principal
bundle. Via associated bundles, any representation of $G_0$ gives rise
to a natural vector bundle on each such manifold. By construction, for
the representation $\frak g_{-1}$, one obtains $\Cal G_0\x_{G_0}\frak
g_{-1}\cong TM$. Likewise, the dual map to the Lie bracket includes
$\frak g_{-2}^*$ as the distinguished line into $\La^2\frak g_{-1}^*$,
so $\Cal G_0\x_{G_0}\frak g^*_{-2}\cong\ell\subset\La^2T^*M$, the
almost conformally symplectic structure underlying the
PACS--structure.

Next, the representation $\La^2\frak g_{-1}^*\otimes\frak g_{-1}$
induces the bundle $\La^2T^*M\otimes TM$ of tangent--bundle--valued
two--forms. The $G_0$--invariant subspace $\ker(\square)$ corresponds
to a smooth subbundle $\ker(\square)\subset\La^2T^*M\otimes TM$, whose
elements will be called \textit{algebraically harmonic}. Likewise,
there is an inclusion
$\La^3_0T^*M\otimes\ell^*\hookrightarrow\La^2T^*M\otimes TM$
corresponding to the algebraic inclusion observed before Theorem
\ref{thm4.3}. This inclusion depends only on the underlying almost
conformally symplectic structure. Indeed, non--degeneracy of $\ell$
implies that $\ell\otimes TM\cong T^*M$, so $TM\cong
T^*M\otimes\ell^*$. Hence $\La^2T^*M\otimes TM\cong\La^2T^*M\otimes
T^*M\otimes\ell^*$, so the above inclusion is obviously there.

\begin{cor}\label{cor4.4}
Suppose that $M$ is endowed with a PACS--structure corresponding to a
contact grading on $\frak g$ (which is not of type $C_n$).

(i) There is a unique linear connection on $TM$ which is compatible
with the PACS--structure and whose torsion $T$ lies in the subspace
$\ker(\square)\oplus (\La^3_0T^*M\otimes\ell^*)$.

(ii) Decomposing $T=T_h\oplus T_a$ according to the direct sum
decomposition in (i), the component $T_a$ is the intrinsic torsion of
the conformally almost symplectic structure underlying the
PACS--structure. In particular, in the case of a PCS--structure, the
canonical connection has algebraically harmonic torsion. 

(iii) In the case of a PCS--structure, the connection on $\ell$
induced by the canonical connection $\nabla$ is flat and its local
parallel sections are exactly those which are closed as two--forms.
\end{cor}
\begin{proof}
(i) follows immediately from Theorem \ref{thm4.3} and the general
  relation between the Spencer differential and compatible connections
  as discussed in Section \ref{2.2}.

For (ii), recall the description of the intrinsic torsion of the
almost conformally symplectic structure from the proof of Proposition
\ref{prop2.3}. This shows that, viewing the value of the torsion in a point as
$\ps\in \La^2\frak g_{-1}^*\otimes\frak g_{-1}^*$, the intrinsic part
of the torsion exactly corresponds to the tracefree part of
$\ga(\ps)$, where $\ga$ is the map from the proof of Theorem
\ref{thm4.3}. Since $\ga$ vanishes on $\ker(\square)$ and restricts to an
injection on $\La^3_0T^*M\otimes\ell\subset\La^2T^*M\otimes TM$, this
implies the result. 

(iii) By assumption $T_a=0$ and hence $T=T_h$ is algebraically
harmonic. But then the fact that $\ker(\square)\subset\ker(\ga)$
implies that for any section $\tau\in\Ga(\ell)$, the bundle map
$i_T\tau$ used in the proof Proposition \ref{prop2.3} has vanishing
alternation. Hence $d\tau$ can be computed as the alternation of
$\nabla\tau$ and the proof of part (4) of Proposition \ref{prop2.3}
applies.
\end{proof}

We will refer to the components $T_h$ and $T_a$ from the second part
as the \textit{harmonic torsion} and the \textit{acs--torsion} of a
PACS--structure.

\begin{remark}\label{rem4.4}
It may happen that both $\ker(\square)$ and $\La^3_0\frak
g_{-1}^*\otimes\frak g_{-2}$ are not irreducible representations of
$G_0$ but decompose into a direct sum of irreducibles. If this is the
case, then one obtains corresponding decompositions of $T_h$ and/or
$T_a$ and there are notions of ``semi--integrability'' or
``semi--torsion--freeness'' available. We will discuss this in examples
below.
\end{remark}

\subsection{Example: (para--)K\"ahler type}\label{4.5} 
The form of $\ker(\square)\subset\La^2T^*M\otimes TM$ can be deduced
from Kostant's theorem, and in most cases, a detailed description is
available in the literature on parabolic contact structures. We will
discuss the PACS--structures of K\"ahler type in more detail, since
they have the strongest connections to well studied structures, and
only briefly comment on the other types. 

Suppose that $M$ carries a PACS--structure of K\"ahler type, and let
$J$ be the corresponding almost complex structure on $M$.  The
harmonic part $\ker(\square)\subset\La^2T^*M\otimes TM$ is determined
in Section 4.2.4 of \cite{book}. It consists of those skew symmetric
bilinear maps, which are of type $(0,2)$, i.e.~which are conjugate
linear (with respect to $J$) in both arguments. Since this subbundle
is induced by an irreducible representation, there is no finer
decomposition of the harmonic torsion $T_h$ available.

On the other hand, the bundle $\La^3_0T^*M\otimes\ell^*$ decomposes
into the sum of two subbundles according to $(p,q)$--types.  (Since we
are dealing with real valued forms, there are just two summands, whose
complexifications split into the sums of types $(3,0)$ and $(1,2)$
respectively $(2,1)$ and $(0,3)$.) So there is a corresponding
decomposition of the acs--torsion into two components.

\begin{prop}\label{prop4.5}
Consider a PACS--structure of K\"ahler type on $M$ corresponding to a
conformal class of almost Hermitian metrics $(g,J)$ on $M$. Then the
harmonic torsion $T_h$ of the geometry coincides (up to a non--zero
constant factor) with the Nijenhuis tensor of the almost complex
structure $J$.

In particular, the canonical connection $\nabla$ of the structure is
torsion free if and only if we deal with a PCS--structure and $J$ is
an integrable complex structure. This is equivalent to the conformal
class locally containing (pseudo--)K\"ahler metrics of the given
signature, which then are unique up to a constant factor. In this
case, the canonical connection coincides with the Levi--Civita
connections of the local K\"ahler metrics in the conformal class.
\end{prop}
\begin{proof}
Compatibility of a linear connection $\nabla$ on $TM$ with the
PACS--structure in particular implies that $J$ is parallel with
respect to the induced connection. But it is well known that this
implies that the $(0,2)$--part of the torsion of $\nabla$ is a
non--zero multiple of the Nijenhuis tensor of $J$. In view of the
description of $\ker(\square)$ given above, this implies the claim on
$T_h$ and that $\nabla$ is torsion free if and only if the structure
is PCS and $J$ is integrable. 

For the last claim, we have already seen in Proposition \ref{prop3.2} that
the PCS--condition is equivalent to the existence of local almost
K\"ahler metrics in the conformal class. But an almost K\"ahler metric
is well known to be K\"ahler if and only if the corresponding almost
complex structure is integrable. We have also seen there that these
local metrics are unique up to a constant factor, so they all have the
same Levi--Civita connection. Since the Levi--Civita connection also
preserves $J$, it preserves the PCS--structure and thus coincides with
$\nabla$.
\end{proof}

The case of PACS--structures of para--K\"ahler type can be analyzed in
a very similar fashion. Apart from $\ell\subset\La^2T^*M$, we have a
decomposition $TM=E\oplus F$ into a direct sum of two
rank--$n$--distributions which are isotropic for $\ell$ in this
case. This gives rise to an isomorphism $\ell\otimes E\cong F^*$. The
harmonic part $\ker(\square)\subset\La^2(E\oplus F)^*\otimes(E\oplus
F)$ is determined in Section 4.3.1 of \cite{book}. It is the direct
sum of two bundles induced by irreducible representations of $G_0$
which correspond to the highest weight bits in $\La^2E^*\otimes F$ and
$\La^2F^*\otimes E$, respectively. Accordingly, the harmonic torsion
decomposes into two pieces $T_h=T_h^E+T_h^F$. A linear connection
$\nabla$ on $TM$ which is compatible with the PACS--structure in
particular preserves the subbundles $E$ and $F$. Using this, one
easily verifies that for $\xi,\eta\in\Ga(E)$, $T_h^E(\xi,\eta)$ is
obtained by projecting $-[\xi,\eta]$ to $F$. So this is exactly the
obstruction to involutivity of the distribution $E$, and similarly for
$T^F_h$.

In particular, we readily see that torsion freeness of the canonical
connection $\nabla$ is equivalent to the structure being PCS and both
$E$ and $F$ being involutive. In this case, $\nabla$ locally coincides
with the Levi--Civita connection of a para--K\"ahler metric. The
acs--torsion in this case splits into four components according
to the analog of $(p,q)$--types.

\subsection{Other examples}\label{4.6}
For a PACS--structure of Grassmannian type, we have an almost
Grassmannian structure $TM\cong E^*\otimes F$, where $E$ and $F$ are
auxiliary bundles of ranks $2$ and $n$, respectively, and an almost
conformally symplectic structure $\ell\subset \La^2E\otimes
S^2F^*\subset\La^2T^*M$. As we have noted in Section \ref{3.3}, this induces
a non--degenerate symmetric bilinear form $b$ on $F$ determined up to
scale, thus defining a signature $(p,q)$. 

The harmonic part $\ker(\square)\subset\La^2T^*M\otimes TM$ for this
case is determined in Section 4.2.5 of \cite{book}. It is a direct sum
of two subbundles, both of which are induced by irreducible
representations. The first of these is isomorphic to $\La^2E\otimes
E^*\otimes S^3_0F^*$, where $S^3_0F^*$ denotes the tracefree part of
the symmetric cube of $F^*$, which, via $b$, sits inside
$S^2F^*\otimes F$. We denote the corresponding component by
$T^\ell_h$. 

The other component is a bit more complicated to describe. It is
contained in $(S^2E^*\otimes E)_0\otimes(\La^2F^*\otimes F)_0$, where
the subscripts indicate tracefree parts. Now the first factor already
corresponds to an irreducible representation of $\frak g_0$, but the
second factor can be included into $\La^2F^*\otimes F^*$ via $b$, and
we have to take the kernel of the resulting alternation map to
$\La^3F^*$. Let us write $T^G_h$ for the corresponding component of
the harmonic torsion.

Now an interesting feature of this case is that, in the case $n>2$,
which we always consider here, an almost Grassmannian structure has an
intrinsic torsion in its own right, see Section 4.1.3 of
\cite{book}. This intrinsic torsion corresponds to the component in
$(S^2E^*\otimes E)_0\otimes(\La^2F^*\otimes F)_0$, which thus is the
same for all linear connections compatible with the almost
Grassmannian structure. In particular, we conclude that the component
$T^G_h$ of the harmonic torsion depends only on the underlying almost
Grassmannian structure and not on $\ell$, and it vanishes if and only
if the structure is Grassmannian. On the other hand, the component
$T^\ell_h$ of the harmonic torsion is a basic invariant for an almost
conformally symplectic structure which is Hermitian with respect to an
almost Grassmannian structure. There is also a finer decomposition of
the acs--torsion $T_a$ available in the Grassmannian case, but we do
not go into this.

The case of PACS--structures of quaternionic type can be dealt with
similarly, with quaternionic linearity and anti--linearity
respectively hermiticity and anti--hermiticity properties replacing
the decompositions coming from the tensor product structure. In
particular, one again obtains one component in the harmonic torsion
which only depends on the underlying almost quaternionic structure and
whose vanishing is equivalent to the structure being quaternionic.

\medskip

For the exceptional PACS--structures, we just give a general
description of $\ker(\square)$. Making things more explicit is a
question of representation theory. As mentioned in Section 4.2.8 of
\cite{book}, a case--by--case inspection shows that $\ker(\square)$
always is induced by an irreducible representation, so there is no
finer decomposition of the harmonic torsion available in the
exceptional cases. In terms of representation theory, this component
can be easily characterized as the Cartan product (i.e.~the
irreducible component of maximal highest weight in the tensor product)
of $\La^2_0(\frak g_{-1})^*$ and $\frak g_{-1}$. Similarly, the
decomposition of the acs--torsion is a question of representation
theory in these cases, and we do not go into details here.

\subsection{Relation to special symplectic connections}\label{4.7}
To conclude the discussion of canonical connections associated to
PACS--structures, we briefly discuss their relation to the theory of
special symplectic connections developed in
\cite{Cahen-Schwachhoefer}. Consider a real simple Lie algebra $\frak
g$, which is not of type $C_n$ and admits a contact grading as
discussed in Section \ref{3.1}. Then we get an inclusion $\frak
g_0\hookrightarrow \frak{csp}(\frak g_{-1})$. As noted in Section
\ref{3.1}, the associated special symplectic subalgebra in the sense
of Cahen--Schwachh\"ofer is $\frak g^0_0:=\frak g_0\cap\frak{sp}(\frak
g_{-1})$.

A crucial ingredient in \cite{Cahen-Schwachhoefer} is that the special
symplectic subalgebra $\frak g_0^0$ determines two spaces of curvature
tensors. First, as for any Lie algebra of matrices, there is the space
$K(\frak g_0^0)$ of formal curvature tensors having values in $\frak
g_0^0$, which is defined as 
$$
\{R\in\La^2(\frak g_{-1})^*\otimes\frak g_0^0:
R(X,Y)(Z)+R(Z,X)(Y)+R(Y,Z)(X)=0\ \forall X,Y,Z\in\frak g_{-1}\}. 
$$
So except for the condition on the values, one just requires the first
Bianchi--identity to hold. Observe that $K(\frak g_0^0)$ naturally
is a $\frak g_0^0$--module. For a special symplectic subalgebra $\frak
g_0^0$, there is a distinguished submodule $\Cal R_{\frak
  g_0^0}\subset K(\frak g_0^0)$. In terms of the contact grading of
$\frak g$, this can be easily described as follows. Choose a non--zero
element $\ps\in\frak g_2$ and then for $A\in \frak g_0^0$ define 
$R_A:\La^2\frak g_{-1}\to\frak g_0$ by 
$$
R_A(X,Y):=[X,[\ps,[A,Y]]]-[Y,[\ps,[A,X]]]. 
$$
Using the Jacobi--identity, one immediately verifies that since
$A\in\frak g_0^0$, $R_A(X,Y)$ always acts trivially on $\frak g_{-2}$
so the values actually lie in $\frak g_0^0$. The Jacobi--identity also
implies that, for $A\in\frak g_0^0$, $R_A$ satisfies the first Bianchi
identity. Hence we have obtained a map $\frak g_0^0\to K(\frak
g_0^0)$, and we denote by $\Cal R_{\frak g_0^0}$ the image of this
mapping. In \cite{Cahen-Schwachhoefer} it is proved that the
Ricci--type contraction maps $\Cal R_{\frak g_0^0}$ to $\frak g_0^0$
and that this contraction vanishes on $R_A$ if and only if
$A=0$. Hence $\Cal R_{\frak g_0^0}$ is isomorphic to $\frak
g_0^0$. 

By construction, $\Cal R_{\frak g_0^0}\subset K(\frak g_0^0)\subset
K(\frak{sp}(\frak g_{-1}))$, so given a symplectic manifold $M$ of
dimension $\dim(\frak g_{-1})$, the space $\Cal R_{\frak g_0^0}$
corresponds to smooth subbundle in $\La^2T^*M\otimes\frak{sp}(TM)$.
Then Cahen--Schwachh\"ofer define a special symplectic connection as a
torsion free connection on a symplectic manifold, whose curvature has
values in this subbundle for some special symplectic subalgebra.

The final ingredient we need for our discussion is a result on the
structure of $K(\frak g_0^0)$ which is proved in Theorem 2.11 of
\cite{Cahen-Schwachhoefer}. The form of the result is closely related 
to the occurrence of Lie algebra cohomology in degree two of
homogeneity zero. This suggests that the result can be deduced from
Kostant's theorem in a way similar to our proof of Theorem \ref{thm4.3},
but since this is not directly related to the topic of this article,
we do not do this but just quote the result. 

\begin{lemma}\label{lem4.7}
If $\frak g$ is not of type $A_n$ or $C_n$ (and hence not of type
$B_2\cong C_2$), then $K(\frak g_0^0)=\Cal R_{\frak g_0^0}$. 
\end{lemma}

Using this, we can now relate  special symplectic connections to
PCS--structures.  

\begin{thm}\label{thm4.7}
Let $\frak g_0^0$ be a special symplectic subalgebra corresponding to
a real simple Lie algebra $\frak g$, which is not of type $C_n$. Then
any special symplectic connection is the canonical connection
associated to a PCS--structure with vanishing harmonic torsion.

For PCS--structures of Grassmannian, quaternionic, and exceptional
type, the converse holds, i.e.~the canonical connection associated to
a torsion--free geometry of that type is a special symplectic
connection. Indeed, these connections are exactly those having special
symplectic holonomy. 

For PCS--structures of K\"ahler or para--K\"ahler type, there is an
additional condition on the curvature of the canonical connection
associated to a torsion--free PCS--structure that has to be satisfied
to obtain a special symplectic connection.
\end{thm}
\begin{proof}
  By the classical Ambrose--Singer theorem, a torsion--free
  connection with curvature contained in $K(\frak g_0^0)$ can be
  obtained from a reduction of structure group to a group with Lie
  algebra $\frak g_0^0$. Via the inclusion $\frak
  g_0^0\hookrightarrow\frak g_0$, one can extend this to a group $G_0$
  with Lie algebra $\frak g_0$, and the connection still is compatible
  with the corresponding reduction. Thus we have obtained a
  PACS--structure which admits a compatible torsion--free
  connection. Hence the structure has to be PCS with vanishing
  harmonic torsion, and the connection has to coincide with the
  canonical one.

For the converse, part (iii) of Corollary \ref{cor4.4} implies that the
curvature of the canonical connection of any PCS--structure has values
in $\frak g_0^0\subset\frak g_0$, so it lies in $K(\frak g_0^0)$. By
Lemma \ref{lem4.7}, we have $K(\frak g_0^0)=\Cal R_{\frak g_0^0}$ in the
cases of Grassmannian, quaternionic, and exceptional type. Thus in
these cases, the connection is special symplectic provided that it is
torsion--free. In the remaining cases, there is a natural $\frak
g_0^0$--invariant complement $\Cal W_{\frak g_0^0}$ to $\Cal R_{\frak
  g_0^0}\subset K(\frak g_0^0)$, and to obtain a special symplectic
connection, one has to require torsion--freeness and vanishing of the
curvature component in $\Cal W_{\frak g_0^0}$.
\end{proof}

As shown in \cite{Cahen-Schwachhoefer}, the additional curvature
condition in the K\"ahler and para--K\"ahler cases turns out to be
vanishing of the so--called Bochner--curvature. Thus, special
symplectic connections in these cases are Levi--Civita connections of
Boch\-ner--K\"ahler metrics respectively Bochner--bi--Langrangean
metrics rather than just of K\"ahler metrics respectively
para--K\"ahler metrics. We will discuss an alternative
characterization of special symplectic connections among the canonical
connections of PCS--structures of (para--)K\"ahler type using
contactifications in \cite{PCS2}.

\end{document}